\documentclass{amsart}

\makeatletter
 \def\LaTeX{\leavevmode L\raise.42ex
   \hbox{\kern-.3em\size{\sf@size}{0pt}\selectfont A}\kern-.15em\TeX}
\makeatother

\newcommand{\BibTeX}{{\rm B\kern-.05em{\sc
i\kern-.025emb}\kern-.08em\TeX}}
\newtheorem{col}{Corollary}[section]
\newtheorem{thm}{Theorem}[section]
\newtheorem{lem}[thm]{Lemma}

\theoremstyle{definition}

\numberwithin{equation}{section}

\begin{document}

\title[  $n$-widths and approximations  on Riemannian  Manifolds]
{  $n$-widths   and Approximation theory on Compact Riemannian Manifolds}

\maketitle
\begin{center}

\author {Daryl Geller}
\footnote{Department of Mathematics, Stony Brook University, Stony Brook, NY 11794-3651; (12/26/1950-01/27/2011)}

\author{Isaac Z. Pesenson }\footnote{ Department of Mathematics, Temple University,
 Philadelphia,
PA 19122; pesenson@temple.edu. The author was supported in
part by the National Geospatial-Intelligence Agency University
Research Initiative (NURI), grant HM1582-08-1-0019. }

\end{center}

\begin{abstract}
We determine upper asymptotic estimates of  Kolmogorov and linear  $n$-widths of unit balls in Sobolev  and Besov norms in $L_{p}$-spaces on 
compact Riemannian 
manifolds. 
The proofs   rely  on  estimates for the near-diagonal localization of the kernels of  elliptic operators. We also summarize some of our previous results about approximations by eigenfunctions of elliptic operators on manifolds.
\end{abstract}



% \keywords{Compact homogeneous manifold, wavelets, Laplace operator, eigenfunctions,}
% \subjclass[2000]{ 43A85; 42C40; 41A17;
%Secondary 41A10}

\section{Introduction and the main results}

%Okay, so we need to restructure the paper so that it is only for homogeneous manifolds, and add some remarks about what survives in the general case.
 %
Daryl Geller and I started to work  on this paper during the Summer of 2010. Sadly, Daryl Geller passed away suddenly in 
January of 2011. I will always remember him as a good friend and a wonderful 
mathematician. 

\bigskip

Approximation theory on compact manifolds is an old  subject \cite{S}, \cite{Sob}, \cite{R}, \cite{Tay}, \cite{Pes79}-\cite{Pes90}, \cite{Ka1}, \cite{Ka2}. However it attracted  considerable interest during last years  \cite{brdai}-\cite{FGS98} due to numerous applications of function theory on $S^{2},\>S^{3},$ and $SO(3)$ to  seismology, weather prediction, astrophysics, texture analysis, signal analysis,  computer vision, computerized tomography, neuroscience, and statistics \cite{BHS}, \cite{FGS98}, \cite{MP}, \cite{P}.

In the classical approximation theory of functions on Euclidean spaces the so called Kolmogorov width $d_{n}$ and linear width $\delta_{n}$ are of primary importance. 
The width $d_{n}$ characterizes the best approximative possibilities by approximations by $n$-dimensional subspaces, the width $\delta_{n}$ characterizes the best approximative possibilities  of any $n$-dimensional linear method. 
The width $d_{n}$ was introduced by A.N.\ Kolmogorov  in \cite{Kol} and $\delta _{n}$ was introduced by V.M.\ Tikhomirov in \cite{Tikh}.

The goal of the paper is of two fold. We  determine asymptotic estimates of  Kolmogorov and linear $n$-widths of unit balls in
Sobolev  and Besov norms in $L_{p}({\bf M})$-spaces on a  compact Riemannian 
manifold ${\bf M}$ and we give a brief account of our previous results about approximations by eigenfunctions of elliptic operators on manifolds.

Let us recall \cite{LGM} that for a given subset $H$ of a normed linear space $Y$, the Kolmogorov $n$-width
$d_{n}(H,Y)$ is defined as
$$
d_{n}(H,Y)=\inf_{Z_{n}}\sup_{x\in H}\inf_{z\in Z_{n}}\|x-z\|_{Y}
$$
where $Z_{n}$ runs over all $n$-dimensional subspaces of $Y$. The linear $n$-width $\delta_{n}(H,Y)$  is defined as 
$$
\delta_{n}(H,Y)=\inf _{A_{n}}\sup_{x\in H}\|x-A_{n}x\|_{Y}
$$
where $A_{n}$ runs over all bounded operators $A_{n}: Y\rightarrow Y$ whose range has dimension $n$.
In our paper 
 the notation $S_{n}$ will be used for either $d_{n},$ or $\delta_{n}$.

One  has the following relation (see \cite{LGM}, pp. 400-403,):
\begin{equation}
\label{pupqdn2}
S_n(H, Y) \leq S_n(H, Y_{1}), \>\>H\subset Y_{1}\subset Y, 
\end{equation}
where  $Y_{1}$ is  a subspace of $ Y$.

If $\gamma \in \bf R$, we write $S_n(H, Y) \ll n^{\gamma}$ to mean that one has the upper estimate $S_n(H, Y) \leq Cn^{\gamma}$ for $n > 0$
where  $C$ is independent  of $n$. 
Let  $L_{q}=L_{q}({\bf M}),\> 1\leq q\leq\infty,$ be the regular Lebesgue  space constructed with the Riemannian density.
Let $L$ be an elliptic smooth second-order 
differential operator $L$ which is self-adjoint and positive definite in $L_{2}({\bf M})$, such as the Laplace-Beltrami operator $\Delta$.   For such an operator all the powers $L^{r}, \>\>r>0,$ 
are well defined on $C^{\infty}({\bf M})\subset L_{2}({\bf M})$ and continuously map $C^{\infty}({\bf M})$ into itself. 
Using duality every operator  $L^{r}, \>\>r>0,$ can be extended  to distributions on ${\bf M}$.
The Sobolev space $W_{p}^{r}=W_{p}^{r}({\bf M}),\> 1\leq p\leq\infty, \>\> r>0,$ is defined as the space of all 
$f\in L_{p}({\bf M}), 1\leq p\leq \infty$ for which the following graph norm is finite
\begin{equation}
\label{Sob}
\|f\|_{W^{r}_{p}({\bf M})}=\|f\|_{p}+\|L^{r/2} f\|_{p}.
\end{equation}
If $p \neq 1, \infty$, this graph norm is independent of $L$, up to equivalence, by elliptic regularity theory on compact manifolds.    If $p = 1$ or $\infty$ we will need to specify which operator $L$ we are using; some of our results will apply for $L$ general.  In fact,
for our results which apply to general ${\bf M}$, we can use any $L$.

Our objective is to obtain asymptotic estimates of $S_{n}(H, L_{q}({\bf M}))$,  where  $H$   is the unit ball $B^r_p({\bf M})$ in the Sobolev space 
$W_{p}^{r}=W_{p}^{r}({\bf M}), 1\leq p\leq\infty,\>\>r>0,$   Thus,
$$
B^r_p=B^r_p({\bf M})=\left\{f\in W_{p}^{r}({\bf M})\>: \> \|f\|_{W_{p}^{r}({\bf M})}\leq 1\right\}.
$$
It is  important to remember that in all our considerations the inequality
 $$
r>s\left(\frac{1}{p}-\frac{1}{q}\right)_+
$$
with $s=dim \>{\bf M}$ will be satisfied. Thus, by the Sobolev embedding theorem the set  $B^{r}_{p}({\bf M})$ is a  subset of $L_{q}({\bf M})$. Moreover, since ${\bf M}$ is compact by the Rellich-Kondrashov theorem the embedding of $B^{r}_{p}({\bf M})$ into $L_{q}({\bf M})$  will be compact.

We set $s = \dim {\bf M}$. Let  as usual $p'=\frac{p}{p-1}$.  Our main result is the following theorem..
\begin{thm}
\label{basic}
(Upper estimate)
For any compact Riemannian manifold, any $L$, and for any $1\leq p,q\leq\infty, \> r>0$,  if  $S_{n}$ is either of $d_{n}$ or $\delta_{n}$ then the following holds 
\begin{equation}
\label{basicway}
S_n(B^r_p({\bf M}),L_q({\bf M})) \ll n^{-\frac{r}{s}+(\frac{1}{p}-\frac{1}{q})_+},
\end{equation}
provided that $-\frac{r}{s}+(\frac{1}{p}-\frac{1}{q})_+$, which we call the {\em basic exponent}, is negative.
\end{thm}

Our results generalize some of the known estimates for
the particular case in which $\bf M$ is a compact symmetric space of rank one; these estimates were obtained in 
papers \cite{BKLT} and \cite{brdai}.
They, in turn generalized and extended results from \cite{BirSol}, \cite{Ho}, \cite{Kas}, \cite{Ma}, \cite{Ka1} and \cite{Ka2}.

Our main Theorems could be carried over to Besov spaces on manifolds using general results about interpolation of compact operators. 

Our main Theorems 
 along with  some general results in \cite{TriebMath}  imply similar results in which balls in Sobolev spaces  $B^r_p({\bf M})$ are replaced by balls  $\mathrm{B}^r_{p,t}({\bf M})$ in appropriate Besov spaces (see section 6).

The proofs of all the main results heavily exploit our estimates for the near-diagonal localization of the kernels of  
elliptic operators on compact manifolds  (see \cite{gpes} and section 2 below for the general case and  \cite{gm1}- \cite{gmmix} for the case of Laplace-Beltrami operator).

  In last section we consider compact homogeneous manifolds and the corresponding Casimir operator $\mathcal{L}$ (see section \ref{hom} for definitions). For this situation we  review our results about approximations by bandlimited functions.
Although  we show  in Theorem \ref{span} that the span of the eigenfunctions of our operator ${\mathcal L}$ is the same as the span 
of all polynomials when one equivariantly embeds the manifold,
the relation between eigenvalues and degrees of polynomials is unknown (at least in 
the general case).  However, it is easy to verify that for compact two-point homogeneous manifolds,  the span of those
eigenfunctions whose eigenvalues are not greater than a value $\ell^{2}, \ \ell\in \mathbb{N},$ is the same as the span of all 
polynomials of degree at most $\ell$. Thus, on compact two-point homogeneous manifolds, our Theorem \ref{approxim}  about approximations by 
bandlimited functions can be reformulated in terms of approximations by polynomials.

\section{Kernels elliptic operators  on compact Riemannian manifolds}

Let $({\bf M},g)$ be a smooth, connected, compact Riemannian manifold without boundary
with  Riemannian measure $\mu$.  We write $dx$ instead of $d\mu(x)$.
For $x,y \in {\bf M}$, let $d(x,y)$ denote the geodesic distance from $x$ to $y$.
We will frequently need the fact that if $M > s$, $x \in {\bf M}$ and $t > 0$, then
\begin{equation}
\label{intest}
\int_{\bf M} \frac{1}{\left[1 + (d(x,y)/t)\right]^M} dy \leq Ct^{s},\>\>\>s = \dim {\bf M},
\end{equation}
with $C$ independent of $x$ or $t$. 
 
Let $L$ be a smooth, positive, second order elliptic differential operator on ${\bf M}$,
whose principal symbol $\sigma_2(L)(x,\xi)$ is positive on  $\{(x,\xi) \in T^*{\bf M}:\ \xi \neq 0\}$.
In the proof of Theorems \ref{basic}  we will take $L$ to be the Laplace-Beltrami
operator of the metric $g$.
We will use the same notation
$L$ for the closure of $L$ from $C^{\infty}({\bf M})$ in $L_{2}({\bf M})$.
In the case $p=2$ this closure is a
self-adjoint positive definite operator on the space $L_{2}({\bf M})$.
The spectrum of this operator, say
$0=\lambda_{0}<\lambda_{1}\leq \lambda_{2}\leq ...$,
is discrete and approaches infinity.  Let
$u_{0}, u_{1}, u_{2}, ...$ be a corresponding
complete system of real-valued orthonormal eigenfunctions, and let
$\textbf{E}_{\omega}(L),\ \omega>0,$ be the span of all
eigenfunctions of $L$, whose corresponding eigenvalues
are not greater than $\omega$.    Since the
operator $L$ is of order two, the dimension
$\mathcal{N}_{\omega}$ of the space ${\mathbf E}_{\omega}(L)$ is
given asymptotically by Weyl's formula,
which says, in sharp form:
For some $c > 0$,
\begin{equation}
\label{Weyl}
\mathcal{N}_{\omega}(L) = c\omega^{s/2} + O(\omega^{(s-1)/2}).
\vspace{.3cm}
\end{equation}
where $s=dim {\bf M}$.
Since $\mathcal{N}_{\lambda_l} = l+1$, we conclude that, for some constants $c_1, c_2 > 0$,
\begin{equation}
\label{lamest}
c_1 l^{2/s} \leq \lambda_l \leq c_2 l^{2/s}  
\end{equation}
for all $l$.
Since $L^m u_l = \lambda_l^m u_l$, and $L^m$ is an elliptic differential
operator of degree $2m$, Sobolev's lemma, combined with the last fact, implies that
for any integer $k \geq 0$, there exist $C_k, \nu_k > 0$ such that
\begin{equation}
\label{ulest}
\|u_l\|_{C^k({\bf M})} \leq C_k (l+1)^{\nu_k}. 
\end{equation}

Suppose  $F \in\mathcal{ S}(\bf{R}^{+})$, the space of restrictions to the 
nonnegative real axis of Schwartz functions on $\bf{R}$.  Using the spectral theorem,
one can define the bounded operator $F(t^{2}L)$ on $L_2({\bf M})$.  In fact, for $f \in L_2({\bf M})$,
\begin{equation}
\label{ft2F}
[F(t^{2}L)f](x) = \int K_t(x,y) f(y) dy,
\end{equation}
where 
\begin{equation}
\label{expout}
K_t(x,y) = \sum_l F(t^2\lambda_l)u_l(x)u_l(y) = K_t(y,x)
\end{equation}
as one sees easily by checking the case $F = u_m$.  
Using (\ref{expout}), (\ref{Weyl}), (\ref{lamest}) and (\ref{ulest}), one easily checks that $K_t(x,y)$ is smooth in $(x,y) \in 
{\bf M} \times {\bf M}$.  We call $K_t$ the kernel of $F(t^2L)$.  $F(t^2L)$ maps $C^{\infty}({\bf M})$
to itself continuously, and may thus be extended to be a map on distributions.  In particular
we may apply $F(t^2L)$ to any $f \in L_p({\bf M}) \subseteq L_1({\bf M})$ (where $1 \leq p \leq \infty$), and by Fubini's theorem
$F(t^2L)f$ is still given by (\ref{ft2F}).  

The following Theorem  about $K_t$ was proved in \cite{gpes}  for general elliptic second order differential self-adoint positive operators.
\begin{thm}
\label{nrdglc}
 Assume $F\in \mathcal{ S}(\bf{R}^{+})$ (the space of restrictions to the 
nonnegative real axis of Schwartz functions on $\bf{R}$).
For $t > 0$, let $K_t(x,y)$ be the kernel of $F(t^{2}L)$.  Then: 
\begin{enumerate}
\item If $F(0) = 0$, then for every pair of
$C^{\infty}$ differential operators $X$ $($in $x)$ and $Y$  $($in $y)$ on ${\bf M}$,
and for every integer $N \geq 0$, there exists $C_{N,X,Y}$  such that for 
$\deg X = j$ and $\deg Y = k$ the following estimate holds 
\begin{equation}
\label{diagest}
t^{s+j+k} \left| \left(\frac{d(x,y)}{t}\right)^{N}XYK_t(x,y)\right| \leq C_{N,X,Y} ,\>\>s=dim\>{\bf M},
\end{equation}
for all $t > 0$ and all $x,y \in {\bf M}$.
\item  For general $F\in \mathcal{ S}(\bf{R}^{+})$ the estimate (\ref{diagest}) at least holds for $0 < t \leq 1$.
\end{enumerate}
\end{thm}

In this article, we will use the following corollaries of the above result.

\begin{col}
\label{kersize}
Assume  $F \in \mathcal{ S}(\bf{R}^{+})$.
For $t > 0$, let $K_t(x,y)$ be the kernel of $F(t^{2}L)$.  Suppose that either:\\
(i) $F(0) = 0$, or\\
(ii) $F$ is general, but we only consider $0 < t \leq 1$.\\ \ \\
Then for some $C > 0$,
\begin{equation}
\label{kersizeway}
|K_t(x,y)| \leq \frac{Ct^{-s}}{\left[ 1+\frac{d(x,y)}{t} \right]^{s+1}},\>\>\>s=dim\>{\bf M},
\end{equation}
for all $t$ and all $x,y \in {\bf M}$.
\end{col}
{\bf Proof}  This is immediate from Theorem \ref{nrdglc}, with $X=Y=I$, if one
considers the two cases $N=0$ and $N=s+1$.

\begin{col}
\label{Lalphok}
Consider  $1 \leq \alpha \leq \infty$, with conjugate index $\alpha'$.
In the situation of Theorem \ref{kersize}, there is a constant $C > 0$ such that
\begin{equation}
\label{kint3a}
\left(\int |K_t(x,y)|^{\alpha}dy\right)^{1/\alpha} \leq Ct^{-s/\alpha'} \ \ \ \ \ \ \ \ \ \ \ \ \ \ \ \ \ \ \ \ \mbox{for all } x, 
\end{equation}
and
\begin{equation}
\label{kint4}
\left(\int |K_t(x,y)|^{\alpha}dx\right)^{1/\alpha} \leq Ct^{-s/\alpha'} \ \ \ \ \ \ \ \ \ \ \ \ \ \ \ \ \ \ \ \ \mbox{for all } y, 
\end{equation}
\end{col}
{\bf Proof}  We need only prove (\ref{kint3a}), since $K_t(y,x) = K_t(x,y)$.

If $\alpha < \infty$,
(\ref{kint3a}) follows from Corollary \ref{kersize}, which tells us that
\[
\int |K_t(x,y)|^{\alpha}dy 
\leq  C \int_{\bf M} \frac{t^{-s\alpha}}{\left[1 + (d(x,y)/t)\right]^{\alpha(s+1)}} dy \leq Ct^{s(1-\alpha)} \]
with $C$ independent of $x$ or $t$, by (\ref{intest}).

If $\alpha = \infty$, the left side of (\ref{kint3a}) is as usual to be interpreted as the $L^{\infty}$ norm
of $h_{t,x}(y) = K_t(x,y)$.  But in this case the conclusion is immediate from 
Corollary \ref{kersize}.  

This completes the proof.\\

We will use Corollary \ref{Lalphok} in conjunction with the following
fact.  We consider operators of the form $f \to {\mathcal K}f$
where
\begin{equation}
\label{kopdf}
({\mathcal K}f)(x) = \int K(x,y)f(y)dy, 
\end{equation}
where the integral is over ${\bf M}$, and where we are using Riemannian measure.
In all applications, $K$ will be continuous on ${\bf M} \times {\bf M}$,
and $F$ will be in $L_1({\bf M})$, so that ${\mathcal K}f$ will be a bounded continuous function.
The following generalization of Young's inequality holds:

\begin{lem}
\label{younggen}
Suppose $1 \leq p, \alpha \leq \infty$, and that $(1/q)+1 = (1/p)+(1/\alpha)$. 
Suppose that $c > 0$, and that
\begin{equation}
\label{kint1}
[\int |K(x,y)|^{\alpha}dy]^{1/\alpha} \leq c \ \ \ \ \ \ \ \ \ \ \ \ \ \ \ \ \ \ \ \ \mbox{for all } x, 
\end{equation}
and
\begin{equation}
\label{kint1}
[\int |K(x,y)|^{\alpha}dx]^{1/\alpha} \leq c \ \ \ \ \ \ \ \ \ \ \ \ \ \ \ \ \ \ \ \ \mbox{for all } y, 
\end{equation}
Then $\|{\mathcal K}f\|_q \leq c\|f\|_p$ for all $f \in L_p$.
\end{lem}

\begin{proof} Let $\beta = q/\alpha \geq 1$,
so that $\beta' = p'/\alpha$.  For any $x$, we have
\begin{eqnarray*}
|({\mathcal K}f)(x)| & \leq & \int |K(x,y)|^{1/\beta'}|K(x,y)|^{1/\beta}f(y)|dy  \\
& \leq & \left(\int |K(x,y)|^{p'/\beta'}dy\right)^{1/p'} \left(\int |K(x,y)|^{p/\beta}|f(y)|^p dy\right)^{1/p}\\
& \leq & c^{1/\beta'}\left(\int |K(x,y)|^{p/\beta}|f(y)|^p dy\right)^{1/p}
\end{eqnarray*}
since $p'/\beta' = \alpha$, $\alpha/p' = 1/\beta'$.  Thus
\begin{eqnarray*}
\|{\mathcal K}f\|^p_q & \leq & c^{p/\beta'}\left(\int \left(\int |K(x,y)|^{p/\beta}|f(y)|^p dy \right)^{q/p}dx \right)^{p/q} \\
& \leq  & c^{p/\beta'}\int \left(\int |K(x,y)|^{pq/\beta p}|f(y)|^{pq/p} dx \right)^{p/q}dy \\
& = & c^{p/\beta'}\int \left(\int |K(x,y)|^{\alpha}dx \right)^{p/q}|f(y)|^p dy \\
& \leq & c^{p/\beta'}c^{p/\beta}\|f\|^p_p
\end{eqnarray*}
as desired.  (In the second line, we have used Minkowski's inequality for integrals.)\\

\end{proof}

\section{Proof of Theorem \ref{basic}}

Now, let $\eta$ be a $C^{\infty}$ function on $[0,\infty)$ which equals $1$ on $[0,1]$,
and which is supported in $[0,4]$.  Define, for $x > 0$,
\[ \phi(x) = \eta(x/4) - \eta(x)\]
so that $\phi$ is supported in $[1,16]$.  For $j \geq 1$, we set 
\[ \phi_j(x) = \phi(x/4^{j-1}).\] 
We also set
$\phi_0 = \eta$, so that $\sum_{j=0}^{\infty} \phi_j \equiv 1$.  We claim:

\begin{lem}
\label{phijest}
(a) If $r > 0$, and $1 \leq p \leq q \leq \infty$, then there is a $C > 0$ such that
\begin{equation}
\label{phijestway}
\|\phi_j(L)f\|_q \leq C(2^{js})^{-\frac{r}{s}+\frac{1}{p}-\frac{1}{q}}\|f\|_{W_p^r({\bf M})}, 
\end{equation}
for all $f \in W_p^r({\bf M})$. In other words, the norm of $\phi_j(L)$, as an element
of   ${\bf B}(W_p^r({\bf M}),L_q({\bf M}))$ (the space of bounded linear operators from $W_p^r({\bf M})$ to $L_q({\bf M})$), is
no more than $C(2^{js})^{-\frac{r}{s}+\frac{1}{p}-\frac{1}{q}}$.\\
(b) Suppose that 
$$
-\frac{r}{s}+\frac{1}{p}-\frac{1}{q} < 0.
$$
  Then $\sum_{j=0}^{\infty} \phi_j(L)$
converges absolutely in ${\bf B}(W_p^r({\bf M}),L_q({\bf M}))$, to the identity operator on $W_p^r({\bf M})$.
\end{lem}
{\bf Proof} (a) 
Define, for $x > 0$,
\[ \psi(x) = \phi(x)/x^{r/2} \]
so that 
$\psi$ is supported in $[1,16]$.  For $j \geq 1$, we set 
\[ \psi_j(x)=\psi(x/4^{j-1}),\] 
which implies   
$$
\phi_j(x) = 2^{-(j-1)r}\psi_j(x)x^{r/2}.
$$
Accordingly, if $f$ is a distribution on ${\bf M}$, for $j \geq 1$,
\[ \phi_j(L)f = 2^{-(j-1)r}\psi_j(L)(L^{r/2}f), \]
in the sense of distributions.   
If now $f \in  W_p^r({\bf M})$, so that $L^{r/2}f \in L_p({\bf M})$, we see from Lemma \ref{Lalphok} with $t = 2^{-j}$, and 
from Lemma \ref{younggen}, that if $(1/q)+1 = (1/p)+(1/\alpha)$, then
\[ \|\phi_j(L)f\|_q \leq C2^{-jr}2^{js/\alpha'}\|L^{r/2}f\|_p \leq C (2^{js})^{-\frac{r}{s}+\frac{1}{p}-\frac{1}{q}}\|f\|_{W_p^r({\bf M})}, \]
as desired.

For (b), we note that by (a), $\sum_{j=0}^{\infty} \phi_j(L)$
converges absolutely in ${\bf B}(W_p^r({\bf M}),L_q({\bf M}))$.  It converges to the identity on smooth functions, hence in the sense
of distributions.  Hence we must have $\sum_{j=0}^{\infty} \phi_j(L) =  I$ in ${\bf B}(W_p^r({\bf M}),L_q({\bf M}))$.  This completes the proof.\\

{\bf Proof of Theorem \ref{basic}}  Since in general $d_n \leq \delta_n$, it suffices to prove the  upper estimate for $\delta_n$.  If $q \leq p$,
then surely $\delta_n(B^r_p({\bf M}),L_q({\bf M})) \leq C\delta_n(B^r_p({\bf M}),L_p({\bf M}))$.  Since the  upper estimate is the same for all $q$ with $q \leq p$, we may
as well assume then that $q = p$.  In short, we may assume $q \geq p$.

Let $\eta$ be the same as above and set $\eta_m(x) = \eta(x/4^{m-1})$ for $m\in \mathbb{N}$.   Then $\sum_{j=0}^{m-1} \phi_j = \eta_m$, which is supported in $[0,4^m]$.  Examining the kernel of 
$\eta_m(L)$ (see (\ref{expout})), we see that 
$$
\eta_m(L): W_p^r({\bf M}) \to {\bf E}_{4^m}(L).
$$
 By Weyl's theorem (\ref{Weyl}), 
there is a positive integer $c$ such that the 
dimension of ${\bf E}_{4^m}(L)$ is at most $c2^{ms}$ for every $m$.  We see then by Lemma \ref{phijest} that

\[ \delta_{c2^{ms}}(B_p^r({\bf M}),L^q({\bf M})) \leq \|I-\eta_m(L)\| \leq \sum_{j=m}^{\infty} \|\phi_j(L)\| \leq 
$$
$$
\sum_{j=m}^{\infty} C(2^{js})^{-\frac{r}{s}+\frac{1}{p}-\frac{1}{q}} \leq C(2^{ms})^{-\frac{r}{s}+\frac{1}{p}-\frac{1}{q}}
\leq  C(c2^{ms})^{-\frac{r}{s}+\frac{1}{p}-\frac{1}{q}}
, \]
where all norms are taken in ${\bf B}(W_p^r({\bf M}),L_q({\bf M}))$.  This proves the basic upper estimate for\\ $n \in A :=
\{ c2^{ms}: m \geq 1\}$.  For any $n \geq c2^s$ we may find $m \in A$ with $m \leq n \leq 2^s m$, and surely
$\delta_n \leq \delta_m$.  This gives the basic upper estimate for all $n$, and completes the proof.\\

\section{Widths of balls in Besov spaces}
   
  The following   definitions of Sobolev and Besov spaces are well known
 \cite{Tay}, \cite{Triebman}.
  Let  $(U_i, \chi_i)$ be a finite atlas on ${\bf M}$ with charts $\chi_i$ 
mapping $W_i$ into the unit ball on ${\bf R}^{n}$, and suppose $\{\zeta_i\}$ is a partition of unity subordinate
to the $U_i$.  The Sobolev space $W_{p}^{r}({\bf M}), \>1\leq p\leq \infty$ and $r$ is natural can be defined as a space of all distributions $f$ on ${\bf M}$ such that
\begin{equation}\label{locSob}
\sum_i \|(\zeta_{i} f) \circ \chi_{i}^{-1}\|_{W_{p}^{r}({\bf R}^{n})} < \infty.
\end{equation}
The Besov space $\mathcal{B}_{p,t}^{\alpha }({\bf M})$ can be defined as a space of distributions $f$ on 
${\bf M}$ for which 
\begin{equation}\label{locBes}
\sum_i \|(\zeta_{i} f) \circ \chi_{i}^{-1}\|_{B_{p,t}^{\alpha }({\bf R}^{n})} < \infty,
\end{equation}
where $\alpha  > 0$, $1 \leq p< \infty$, and $0 < t < \infty$ and $B_{p,t}^{\alpha }({\bf R}^{n})$ is the regular Besov space. This definition does not depend on the choice of charts or partition of unity (\cite{Triebman}).\\

An important property of  Besov spaces
$\mathcal{B}^{\alpha}_{p,t}({\bf M}), \alpha>0, 1\leq  p<\infty, 1\leq t\leq \infty,$ is that they can be described 
using Peetre's interpolation $K$-functor \cite{BL}, \cite{KPS}, \cite{Trieb1}.  Namely, 
\begin{equation}
\mathcal{B}^{\alpha}_{p,t}({\bf M})=\left(L_{p}({\bf M}),W^{r}_{p}({\bf M})\right)^{K}_{\alpha/r,q},
\label{Besovnorm}
\end{equation}
where $r$ can be any natural  such that $0<\alpha<r, 1\leq
t<\infty$, or  $0\leq\alpha\leq r,t= \infty$.
   Since ${\bf M}$ is compact by the Rellich-Kondrashov theorem the embedding of      the ball $B^{r}_{p}({\bf M})$ into $L_{q}({\bf M})$ is compact as long as the condition
 \begin{equation}
r>s\left(\frac{1}{p}-\frac{1}{q}\right)_+
\end{equation}
 is satisfied.
  By an interpolation theorem for compact operators (\cite{TriebMath}, Theorem 1.16.2) the embedding into $L_{q}({\bf M})$ of the unit  ball in the corresponding Besov space $\mathcal{B}^{\alpha}_{p,t}({\bf M})$ is also compact.

These facts allow us to use  some general results in \cite{TriebMath} (Theorem 1.16.3) about interpolation of compact operators which along with our main results produce similar theorems  about  balls  $\mathrm{B}^r_{p,t}({\bf M})$ in appropriate Besov spaces.

\begin{thm} 

Let $ {\bf M}$ be a compact Riemannian manifold.  For every choice of parameters $\>\>1\leq p<\infty, \>\>1\leq q\leq \infty, \>\>r>0,$ for which   the following relation holds 
  $$
d_n(B^r_p({\bf M}), L_q({\bf M}))\ll n^{\gamma},
$$ 
for the Kolmogorov $n$-width of the unit ball $B^r_p({\bf M})$ in the Sobolev space $W_{p}^{r}({\bf M})$ then the similar relation holds for the Kolmogorov $n$-width of the unit ball $\mathrm{B}^r_{p,t}({\bf M})$ in the Besov  space $\mathcal{B}_{p,t}^{r}({\bf M})$ i.e.
$$
d_n(\mathrm{B}^r_{p,t}({\bf M}), L_q({\bf M}))\ll n^{\gamma}.
$$ 
\end{thm}

\section{Approximation theory on compact homogeneous manifolds}\label{hom}

\subsection { Compact homogeneous manifolds}\label{homman}

The most complete results will be obtained for compact homogeneous manifolds.

A {\it homogeneous compact manifold} $M$ is a
$C^{\infty}$-compact manifold  on which a compact
Lie group $G$ acts transitively. In this case $M$ is necessary of the form $G/K$,
where $K$ is a closed subgroup of $G$. The notation $L_{2}(M),$ is used for the usual Hilbert spaces,  with  invariant measure  $dx$ on $M$.

The Lie algebra $\textbf{g}$ of a compact Lie group $G$ %hgfei
is then a direct sum
$\textbf{g}=\textbf{a}+[\textbf{g},\textbf{g}]$, where
$\textbf{a}$ is the center of $\textbf{g}$, and
$[\textbf{g},\textbf{g}]$ is a semi-simple algebra. Let $Q$ be a
positive-definite quadratic form on $\textbf{g}$ which, on
$[\textbf{g},\textbf{g}]$, is opposite to the Killing form. Let
$X_{1},...,X_{d}$ be a basis of
$\textbf{g}$, which is orthonormal with respect to $Q$.
 Since the form $Q$ is $Ad(G)$-invariant, the operator
\begin{equation}\label{Casimir}
-X_{1}^{2}-X_{2}^{2}-\    ... -X_{d}^{2},    \ d=dim\ G
\end{equation}
is a bi-invariant operator on $G$, which is known as the {\it Casimir operator}.
This implies in particular that the corresponding operator on $L_{2}(M)$,
\begin{equation}\label{Casimir-Image}
\mathcal{L}=-D_{1}^{2}- D_{2}^{2}- ...- D_{d}^{2}, \>\>\>
       D_{j}=D_{X_{j}}, \        d=dim \ G,
\end{equation}
commutes with all operators $D_{j}=D_{X_{j}}$.
The operator $\mathcal{L}$, which is usually called the {\it Laplace operator}, is
the image of the Casimir operator under differential of quazi-regular representation in $L_{2}(M)$. It is important to realize that in general, the operator $\mathcal{L}$ is not necessarily the {\it Laplace-Beltrami operator} of the natural  invariant metric on $M$. But it coincides with such operator at least in the following cases:
1) If $M$ is a $d$-dimensional torus, 2) If the manifold $M$ is itself a compact  semi-simple Lie group group $G$ (\cite{H2}, Ch. II), 3) If $M=G/K$ is a compact symmetric space of
  rank one (\cite{H2},
Ch. II, Theorem 4.11).

In the case of a compact manifold the norm ({\ref{locSob}) of the Sobolev space $W_{p}^{r}({\bf M}), \>\>1\leq p\leq \infty,$ $ r\in \mathbb{N}$, is equivalent to one of the following norms \cite{Pes88a}
$$
\|f\|_{p}+\sum _{1\leq k\leq r}\sum_{1\leq i_{1},...,i_{k}\leq d}\|D_{i_{i}}...D_{i_{k}}f\|_{p}\sim \|f\|_{p}+\sum_{1\leq i_{1},...,i_{r}\leq d}\|D_{i_{i}}...D_{i_{r}}f\|_{p},
$$ 
where $d=dim \>G$.

\subsection{Bernstein spaces on compact homogeneous manifolds}
Returning to the compact homogeneous manifold ${\bf M}=G/K$,   let $\mathbb{D}=\{D_{1},...,D_{d}\},\>\>d=\dim G, $ be the same set of operators as in (\ref{Casimir-Image}). Let us define the {\it Bernstein space}
$$
\textbf{B}_{\omega}^{p}(\mathbb{D})=\{f\in L_{p}({\bf M}):
 \|D_{i_{1}}...D_{i_{k}}f\|_{p}\leq
 \omega^{k}\|f\|_{p}, \>\>1\leq i_{1},...i_{k}\leq d,\>\omega\geq 0\}
 $$
where $d=dim\> G$.

As before, the notation ${\bf E}_{\omega}(\mathcal{L}),\>\>\omega\geq 0, $ will be used for a span of eigenvectors of $\mathcal{L}$ with eigenvalues $\leq \omega$. 
For these spaces the next two theorems hold (see  \cite{Pes90}, \cite{Pes08}):
\begin{thm}
The following properties hold:
\begin{enumerate}

\item $$\textbf{B}_{\omega}^{p}(\mathbb{D})=\textbf{B}_{\omega}^{q}(\mathbb{D}),\>\>\>1\leq p\leq q\leq\infty,\>\>\omega\geq 0.$$

\item
$$
\textbf{B}^{p}_{\omega}(\mathbb{D})\subset {\bf E} _{\omega^{2}d}(\mathcal{L})\subset
\textbf{B}^{p}_{\omega\sqrt{d}}(\mathbb{D}), \>\>\>d=\dim\> G,\>\>\>\omega\geq 0.
$$

\item
$$
\|\mathcal{L}^{k}\varphi\|_{q}\leq C({\bf M})
\omega^{2k+\frac{d}{p}-\frac{d}{q}}\|\varphi\|_{p},\>\>\> \varphi\in {\bf E}_{\omega}(\mathcal{L}),\>\>\>k\in
\mathbb{N},
$$
where $ d=\dim \>G, \>\>1\leq p\leq q\leq\infty$.
\end{enumerate}
\end{thm}

Every compact Lie
group can be considered to be a closed subgroup of the orthogonal
group $O(\mathbb{R}^{N})$ of some Euclidean space
$\mathbb{R}^{N}$.  It means that  we can identify ${\bf M}=G/K$ with the orbit
of a unit vector $v\in \mathbb{R}^{N}$ under the action of a subgroup
of the orthogonal group $O(\mathbb{R}^{N})$  in some $\mathbb{R}^{N}$.
In this case $K$ will be  the stationary group of $v$. Such an
embedding of ${\bf M}$ into $\mathbb{R}^{N}$ is called {\it equivariant}.

We choose an orthonormal  basis in $\mathbb{R}^{N}$ for which the
first vector is the vector $v$: $e_{1}=v, e_{2},...,e_{N}$. Let $
\textbf{P}_{r}({\bf M}) $ be the space of restrictions to ${\bf M}$ of all
polynomials in $\mathbb{R}^{N}$ of degree $r$. This space is
closed in the norm of $L_{p}({\bf M}), 1\leq p\leq \infty,$ which is
constructed with respect to the $G$-invariant normalized measure on ${\bf M}$ \cite{Pes90}, \cite{Pes08}.

\begin{thm}\label{span} If ${\bf M}$ is embedded into an $R^{N}$ equivariantly, then
$$
\textbf{P}_{r}({\bf M})\subset \textbf{B}_{r}(\mathbb{D})\subset
 {\bf E}_{r^{2}d}(\mathcal{L})\subset
\textbf{B}_{r\sqrt{d}}(\mathbb{D}), \>\>\>d=dim \>G,\>\>\>r\in \mathbb{N},
$$
and
$$
span_{r\in \mathbb{N}}\>\textbf{P}_{r}({\bf M})=span_{\omega\geq 0}
\>\textbf{B}_{\omega}(\mathbb{D})=span_{j\in \mathbb{N}}\>
{\bf E}_{\lambda_{j}}(\mathcal{L}).
$$
\end{thm}

\subsection{Besov spaces on compact homogeneous manifolds}

For the same operators as above $D_{1},...,D_{d},\ d=dim \ G$,  (see section 3) let $T_{1},..., T_{d}$
be the corresponding one-parameter groups of translation along integral
curves of the corresponding vector  fields i.e.
 \begin{equation}
 T_{j}(\tau)f(x)=f(\exp \tau X_{j}\cdot x),\>
 x\in \bold{M}=G/K,\> \tau \in \mathbb{R},\> f\in L_{p}(\bold{M}),\> 1\leq p< \infty,
 \end{equation}
 here $\exp \tau X_{j}\cdot x$ is the integral curve of the vector field
 $X_{j}$ which passes through the point $x\in \bold{M}$.
 The modulus of continuity is introduced as
\begin{equation}
\Omega_{p}^{r}( s, f)= $$ $$\sum_{1\leq j_{1},...,j_{r}\leq
d}\sup_{0\leq\tau_{j_{1}}\leq s}...\sup_{0\leq\tau_{j_{r}}\leq
s}\|
\left(T_{j_{1}}(\tau_{j_{1}})-I\right)...\left(T_{j_{r}}(\tau_{j_{r}})-I\right)f\|_{L_{p}(\bold{M})},\label{M}
\end{equation}
where $d=\dim \>G,\>f\in L_{p}(\bold{M}),1\leq p< \infty,
\ r\in \mathbb{N},  $ and $I$ is the
identity operator in $L_{p}(\bold{M}).$   We consider the space of all functions in $L_{p}(\bold{M})$ for which the
following norm is finite:
\begin{equation}
\|f\|_{L_{p}(\bold{M})}+\left(\int_{0}^{\infty}(s^{-\alpha}\Omega_{p}^{r}(s,
f))^{t} \frac{ds}{s}\right)^{1/t} , 1\leq p,t<\infty,\label{BnormX}
\end{equation}
with the usual modifications for $t=\infty$.
The following theorem is a rather particular case of general results that can be found in \cite{Pes79}, \cite{Pes88a}.
 
 \begin{thm}
 
 If ${\bf M}=G/K$ is a compact homogeneous manifold  the norm of the Besov space $\mathcal{B}^{\alpha}_{p,t}({\bf M}),
 0<\alpha<r\in \mathbb{N},\ 
1\leq p, t< \infty,$ is equivalent to the norm (\ref{BnormX}). Moreover, the norm
(\ref{BnormX}) is equivalent to the norm
\begin{equation}
\|f\|_{W_{p}^{[\alpha]}(\bold{M})}+\sum_{1\leq j_{1},...,j_{[\alpha] }\leq d}
\left(\int_{0}^{\infty}\left(s^{[\alpha]-\alpha}\Omega_{p}^{1}
(s,D_{j_{1}}...D_{j_{[\alpha]}}f)\right)^{t}\frac{ds}{s}\right)^{1/t},d=dim\>G,\label{nonint}
\end{equation}
if $\alpha$ is not integer ($[\alpha]$ is its integer part).  If
$\alpha=k\in \mathbb{N}$ is an integer then the norm
(\ref{BnormX}) is equivalent to the norm (Zygmund condition)
\begin{equation}
\|f\|_{W_{p}^{k-1}(\bold{M})}+ \sum_{1\leq j_{1}, ... ,j_{k-1}\leq d }
\left(\int_{0}^{\infty}\left(s^{-1}\Omega_{p}^{2}(s,
D_{j_{1}}...D_{j_{k-1}}f)\right)
 ^{t}\frac{ds}{s}\right)^{1/t}, d=dim\>G.\label{integer}
\end{equation}
\end{thm}
For $1 \leq p \leq \infty$ we define a measure of the best approximation by functions in ${\bf E}_{\omega}(\mathcal{L})$ as
$$
\mathcal{E}(f,\omega,p)=\inf_{g\in
{\bf E}_{\omega}(\mathcal{L})}\|f-g\|_{L_{p}({\bf M})} \, \, \mbox{for} \, \,  f \in L_p({\bf M}).
$$
The following theorem was proved in \cite{Pes88}, \cite{Pes09}, \cite{gpes}.
% We then have \cite{gp}:
\begin{thm}
\label{approxim}
Suppose that $\alpha > 0, 1 \leq p \leq \infty$, and $0 < t< \infty$.  Then the norm of the Besov space  
$\mathcal{ B}^{\alpha }_{p,t}({\bf M})$  is equivalent to the following one
\begin{equation}
\label{errgdpq}
\|f\|_{{\mathcal{B}^{\alpha}_{p,t}({\bf M})}} :=
\|f\|_{L_p({\bf M})} + \left
(\sum_{j=0}^{\infty} \left [2^{\alpha j}{\mathcal E}(f, 2^{2j},p) \right ]^t
\right )^{1/t} < \infty.
\end{equation}
\end{thm}


\begin{thebibliography}{99}

 
\bibitem{BL}
J. ~Bergh, J. ~Lofstrom, {\em Interpolation spaces},
Springer-Verlag, 1976.


\bibitem{BHS}
Bernstein, S., Hielscher, R., Schaeben, H., {\em The generalized totally geodesic
Radon transform and its application to texture analysis}, Math. Meth. Appl.
Sci., 32:379--394 (2009)
 
\bibitem{BirSol} M. S. Birman, M. Z.  Solomjak,  
{\em Piecewise polynomial approximations of functions of classes $W_{p}^{ \alpha}$ }, (Russian) 
Mat. Sb. (N.S.) 73 (115) 1967 331-355 

\bibitem{brdai} G. Brown and F. Dai (2005), {\em Approximation of smooth functions on compact
two-point homogeneous spaces}, J. Func. Anal. {\bf 220} (2005), 401-423

\bibitem{BKLT}
B. Bordin, A.K. Kushpel, J. Levesley, S.A. Tozoni, {\em Estimates of n-widths of Sobolev  classes
on compact globally symmetric spaces of rank one}, J. Funct. Anal. 202 (2) (2003) 307-326.

\bibitem{BDS} 
G. Brown, F. Dai, Sun Yongsheng, {\em Kolmogorov width of classes of smooth functions on the
sphere} , J. Complexity 18 (4) (2002) 1001-1023.


\bibitem{FGS98}
 W.~Freeden, T.~Gervens, M.~Schreiner,
  {\em Constructive approximation on the  spheres. With applications to
  geomathematics}, Numerical Mathematics and Scientific Computation,
  The Claredon Press, Oxford University Press,
  New York, 1998.

\bibitem{gm1} D. Geller and A. Mayeli, \textit{Continuous wavelets and frames on stratified Lie
groups I}, Journal of Fourier Analysis and Applications, {\bf 12} (2006), 543-579.
\bibitem{gmcw} D. Geller and A. Mayeli, {\em Continuous Wavelets on Compact Manifolds},  Math. Z. {\bf 262} (2009), 895-927.
\bibitem{gmfr} D. Geller and A. Mayeli, {\em Nearly Tight Frames and Space-Frequency Analysis on Compact Manifolds} (2009), 
Math. Z. {\bf 263} (2009), 235-264.
\bibitem{gmmix} D. Geller and D. Marinucci, {\em Mixed needlets} ,J. Math. Anal. Appl. 375 (2011), no. 2, 610Ð630.

\bibitem{gpes} D. Geller and I. Pesenson, {\em Bandlimited localized Parseval frames and Besov spaces on compact homogeneous manifolds}, J. Geom. Anal. 21 (2011), no. 2, 334Ð371.

\bibitem{glus} E.D. Gluskin, {\em Norms of random matrices and diameters of finite-dimensional sets}, Math. Sb.
{\bf 120} (1983), 180-189.
\bibitem {H2}
S. ~Helgason, {\em Groups and Geometric Analysis}, Academic Press,
1984.
\bibitem{Ho} K. Hollig, {\em Approximationszahlen von Sobolev-Einbettungen}, Math. Ann. 242 (3) (1979) 27-
281 (in German).


\bibitem{Ka1} A.I. Kamzolov, {\em The best approximation of the classes of functions $ W_{
p}(S^{d-1})$ by polynomials
in spherical harmonics}, Math. Notes 32 (1982) 622-626.

\bibitem{Ka2} A.I. Kamzolov, {\em On the Kolmogorov diameters of classes of smooth functions on a sphere},
Russian Math. Surveys 44 (5) (1989) 196-197.

\bibitem{Kas} B.S. Kashin, {\em The widths of certain finite-dimensional sets and classes of smooth functions}, Izv.
Akad. Nauk SSSR 41 (1977) 334-351.


\bibitem{Kol} A. Kolmogoroff, {\em Uber die beste Annaherung von Functionen einer gegebenen Funktionenclasse}, Ann. Math. {\bf 37}, (1936), 107-110.



\bibitem{KPS}
S. ~ Krein, Y. ~Petunin, E. ~Semenov, {\em  Interpolation of
linear operators}, Translations of Mathematical Monographs, 54.
AMS, Providence, R.I., 1982.

\bibitem{LGM} G.G. Lorentz, M.V. Golitschek, Yu. Makovoz, {\em Constructive Approximation (Advanced Problems)},
Springer, Berlin, 1996.

\bibitem{Ma}V.E. Maiorov, {\em Linear diameters of Sobolev classes}, Dokl. Akad. Nauk SSSR 243 (5) (1978),1127-1130 (in Russian).
   
\bibitem{MP}
D. Marinucci,  G. Peccati, {\em  Random fields on the sphere. Representation, limit theorems and cosmological applications}, London Mathematical Society Lecture Note Series, 389. Cambridge University Press, Cambridge, 2011. xii+341 pp. ISBN: 978-0-521-17561-6.
\bibitem{Pes79}
I. ~Pesenson, {\em Interpolation spaces on Lie groups}, (Russian)
Dokl. Akad. Nauk SSSR 246 (1979), no. 6, 1298--1303.
 
 \bibitem{Pes88a}
I.~Pesenson, {\em On the abstract theory of Nikolski-Besov spaces}, (Russian) Izv. Vyssh. Uchebn. Zaved. Mat. 1988, no. 6, 59--68; translation in Soviet Math. (Iz. VUZ) 32 (1988), no. 6, 80-92
    
\bibitem{Pes88}
I.~Pesenson, {\em The Best Approximation in a Representation Space
of a Lie Group}, Dokl. Acad. Nauk USSR, v. 302, No 5, pp.
1055-1059, (1988). Engl. Transl. in Soviet Math. Dokl., v.38, No
2, pp. 384-388, 1989.



\bibitem{Pes90}
 I. ~Pesenson, {\em  The Bernstein Inequality in the Space of
Representation of a Lie group}, Dokl. Acad. Nauk USSR {\bf 313}
(1990), 86--90;
 English transl. in Soviet Math. Dokl. {\bf 42} (1991).


     
\bibitem{Pes00}
I. ~Pesenson, {\em A sampling theorem on homogeneous manifolds},
Trans. Amer. Math. Soc. {\bf 352} (2000), no. 9, 4257--4269.

\bibitem{Pes04a}
I.~Pesenson, {\em An approach to spectral problems on Riemannian
manifolds,}  Pacific J. of Math. Vol. 215(1), (2004), 183-199.

\bibitem{Pes04b}
I.~Pesenson,  {\em Poincare-type inequalities and reconstruction
of Paley-Wiener functions on manifolds, }  J. of Geometric Analysis
, (4), 1, (2004), 101-121.

\bibitem{Pes08}
 I. ~Pesenson, {\em   Bernstein-Nikolski inequality and Riesz interpolation Formula on
 compact homogeneous manifolds,} J. Approx. Theory ,150, 
 (2008). no. 2, 175-198.
 
 

\bibitem{Pes09}
 I. ~Pesenson, {\em Paley-Wiener approximations and multiscale approximations in
Sobolev and Besov spaces on manifolds,}  J. of Geometric Analysis, 4, (1), (2009), 101-121.



\bibitem{P}
G. PeyrŽ, {\em Manifold models for signals and images}, Computer Vision and Image Understanding, 113 (2009) 249-260.

 \bibitem{pin} A. Pinkus, {\em $n$-widths in Approximation Theory}, Springer, New York, 1985.

\bibitem{R}
 D. ~Ragozin,
 {\em Polynomial approximation on compact manifolds and homogeneous
 spaces},
 Trans. Amer. Math. Soc. 150 (1970), 41--53.
 
 
\bibitem{S} I.J.~Schoenberg, {\em Positive definite functions on spheres},
 Duke. Math.J., 9(1942), 96-108.


\bibitem{Sob}
S. L. Sobolev, {\em Cubature formulas on the sphere invariant under finite groups of rotations},
Soviet Math. 3 (1962), 1307-1310.


  \bibitem{Tay}M. Taylor, {\em Fourier series on compact Lie groups},  Proc. Amer. Math. Soc. 19 1968 1103-1105.
 
\bibitem{Tikh} V.M. Tikhomirov, {\em Diameters of sets in functional spaces and the theory of best approximations},  Uspehi Mat. Nauk  15  no. 3 (93) 81--120 (Russian); translated as Russian Math. Surveys  15  1960 no. 3, 75--111.


\bibitem{TriebMath} 
H. ~Triebel, {\em Interpolation Theory, Function
Spaces, Differential Operators}, North-Holland Mathematical Library, 18. North-Holland Publishing Co., Amsterdam-New York, 1978. 528 pp.

 
\bibitem{Triebman} H. Triebel, {\em Spaces of Besov-Hardy-Sobolev type on complete Riemannian manifolds}, Ark.
Mat., 24,  (1986), 299-337.

\bibitem{Trieb1} 
H. ~Triebel, {\em  Theory of function spaces II,}
  Monographs in Mathematics, 84. Birkhauser Verlag, Basel, 1992.
  

   \end{thebibliography}
\end{document}